\documentclass[12pt]{amsart}
\usepackage{amsmath,amsfonts,amssymb}
\numberwithin{equation}{section}

\newtheorem{proposition}{Proposition}[section]

\DeclareMathOperator{\Ad}{Ad}

\title [Berezin symbols...]{Berezin symbols and spectral measures of representation operators }
\author {Benjamin Cahen}
\address{Universit\'e de Lorraine, Site de Metz, UFR-MIM,
D\'epartement de math\'ematiques,
B\^atiment A,
3 rue Augustin Fresnel, BP 45112,
57073 METZ Cedex 03, France}
\email{benjamin.cahen@univ-lorraine.fr}


\begin{document}

\maketitle

\begin{abstract}
Let $G$ be a Lie group with Lie algebra $\mathfrak g$ and
let $\pi$ be a unitary representation of $G$ realized on a reproducing kernel Hilbert space. We use Berezin quantization to study spectral measures associated with operators $-id\pi(X)$ for $X\in {\mathfrak g}$.
As an application, we show how results about contractions of Lie group representations give rise to results on convergence of sequences of spectral measures. We give some
examples including contractions of $SU(1,1)$ and $SU(2)$ to the Heisenberg group. 
\end{abstract}

\vspace{1cm}

\section {Introduction} \label{sec:1}

In the years 1970-1975, a general theory of quantization on homogeneous 
K\"ahler manifolds was developed by F. A. Berezin in \cite{Be}, \cite{Be1},  \cite{Be2}. In this theory, an important tool is the notion of covariant symbol
of an operator acting on a reproducing kernel Hilbert space of square-integrable holomorphic functions on a K\"ahler manifold \cite{Be1}. In fact, this notion of covariant symbol has its own interest and, in particular, it appeared
early that Berezin symbols could be helpful to study spectral properties of operators on reproducing kernel Hilbert spaces, see \cite{Be}.

Let $G$ be a Lie group and let $\pi$ be a unitary representation of $G$ on a
reproducing kernel Hilbert space $\mathcal H$ consisting of functions on a 
homogeneous $G$-manifold $M$. Denote by $\mathfrak g$ the Lie algebra of $G$ and by $d\pi$ the differential of $\pi$. In this paper, we are concerned 
with self-adjoint operators of the form $-id\pi(X)$ 
where $X\in {\mathfrak g}$. When $G$ is a simple Lie group and $X$ generates a non-compact one-parameter subgroup of $G$ then C. C. Moore
has shown that the projection-valued measure $E_{\lambda}$ associated
with $-id\pi(X)$ is absolutely continuous with respect to the Lebesgue measure and that the spectrum of $-id\pi(X)$ is either $\mathbb R$ or half-line $(0, \infty)$ or $(-\infty,0)$, see \cite{Moo}. Moreover, S. C. Scull proved
that the spectrum of $-id\pi(X)$ is $\mathbb R$ unless $G$ is the group of a bounded symmetric domain and, in that case, half-line occurs only for $\pi$
in the holomorphic discrete series of $G$ and for certain $X$, see \cite{Sc1},
\cite{Sc2}.

The aim of this paper is to show how Berezin symbols can be used to study
the spectral distribution $d(\langle E_{\lambda}f,f\rangle)$ of $-id\pi(X)$ for some $X\in {\mathfrak g}$ and $f\in {\mathcal H}$. We give in particular some integral formulas for spectral measures allowing explicit computations in some cases. This is illustrated by various examples including the non-degenerate unitary irreducible representations of the Heisenberg group
as well as the holomorphic discrete series representations of $SU(1,1)$.
In particular, we recover some results of \cite{Luo} in a much simpler way.

Otherwise, in the series of papers \cite{{CaBe1}, {CaLux}, {CaJAM}, {CaCont}, {CaDS}}, we used Berezin symbols to obtain some results about contractions of Lie group representations. Recall that if a Lie group $G_0$ is the contraction of a Lie group $G_1$, that is, $G_0$ is the limit case of a sequence of Lie groups isomorphic to $G_1$, then it often happens that the unitary irreducible representations of $G_0$ are also limits in some sense of
sequences of unitary irreducible representations of $G_1$, see \cite{IW}, \cite{Doo}. Moreover, it appears that such contractions of Lie group representations can be connected to convergence of Berezin (covariant) symbols of representation operators, see for instance \cite{{CaJAM}, {CaCont}, {CaDS}}. Thus by combining results on contractions of Lie group representations with the above mentioned formulas for spectral measures of representation operators, we obtain here some results about convergence of these spectral measures. In particular we consider the contraction of the discrete series representations of $SU(1,1)$ and the contraction of the unitary irreducible representations of $SU(2)$ to unitary irreducible representations of the Heisenberg group. Of course, we can hope for further results concerning unitary representations of other Lie groups.

This paper is organized as follows. In Section \ref{sec:2} and 
Section \ref{sec:3}, we recall some basic facts on Berezin symbols of operators acting on reproducing kernel Hilbert spaces. Integral formulas for
spectral measures of $-id\pi(X)$ are given in Section \ref{sec:4} and then illustrated in Section \ref{sec:5} by the case of the Heisenberg group. Finally,
applications of contraction results to spectral measures are presented in 
Section \ref{sec:6} and Section \ref{sec:7}.

\section {Generalities on Berezin quantization} \label{sec:2}
In this section, we review some facts on Berezin
quantization \cite{Be1}, \cite{Be2}. We follow more or less the presentation
of \cite{BH} and \cite{CaScand}.

Let $G$ be a Lie group and let $M$ be a $G$-homogeneous space.
Let $\mu$ be a $G$-invariant measure on $M$. Let $K$ be a measurable function
on $M$ such that $K(x)>0$ almost everywhere and let $\tilde \mu$ be
the measure on $M$ defined by $d\tilde {\mu}(x)=K(x)^{-1}d\mu(x)$.
Let $\mathcal H$ be a reproducing kernel Hilbert space of square
integrable functions on $M$ with respect to $\tilde \mu$. This means that
$\mathcal H$ is a Hilbert space with respect to the $L^2$-norm and, for
each $x\in M$, the evaluation map ${\mathcal H}\ni f \rightarrow f(x)$
is continuous. Then, for each $x\in M$,
there exists a unique function $e_x\in {\mathcal
H}$ (called \textit {a coherent state}) such that $f(x)=\langle f,e_x\rangle$ for each $f\in {\mathcal H}$. The function
$k(x,y):=\overline {e_x(y)}=\langle e_y,e_x\rangle$ is then called \textit {the reproducing
kernel} of ${\mathcal H}$.

Let 
$\alpha: \,G\times M\rightarrow {\mathbb C}^{\ast}$ be a function such that \begin{equation*}\alpha (g_1 g_2,x)=\alpha
(g_1,g_2\cdot x)\alpha (g_2, x)\end{equation*}
for each 
$g_1,g_2 \in G$ and $x\in M$. Then we can define  an action $\pi$ of $G$ on the
space of all functions on $M$, according to the formula
\begin{equation*}(\pi(g)f)(x)=\alpha (g^{-1},x)\,f(g^{-1}\cdot x). \end{equation*}
Assume that $\pi(g)(f)\in {\mathcal H}$ for each $g\in G$
and $f\in {\mathcal H}$. Then $\pi$ induces a representation of $G$
on ${\mathcal H}$. 
\begin{proposition}  \label{prop:21}\cite{BH}, \cite{CaScand}.
\begin{enumerate} 

\item  If $\alpha$ and $K$ are compatible in the sense that we have \begin{equation*}K(g\cdot x)=\vert \alpha
(g,x)\vert^{-2}K(x), \qquad g\in G, \,x\in M, \end{equation*}
then the representation $\pi$ is unitary.

\item If the representation $\pi$ is unitary then we have
\begin{equation*}\pi(g)e_x=\overline {\alpha (g,x)}e_{g\cdot x},\quad \qquad g\in G,\,x\in M, \end{equation*}
and
\begin{equation*}k(g\cdot x,g\cdot y)=\alpha (g,x)^{-1}\overline
{\alpha (g,y)}^{-1}k(x,y),\qquad g\in G,\,x,\,y \in
M.\end{equation*}

\item Moreover, in this case, $\mu$
can be normalized so that $k(x,x)=K(x)$ ($x\in M$).
\end{enumerate}

\end{proposition}

In the rest of the section, we assume that the conditions introduced in the previous proposition are satisfied.

Now, let $A$ be an operator on $\mathcal H$. The Berezin
(covariant) symbol of $A$ is the function $S(A)$ defined on $M$ by
\begin{equation} S(A)(x)=\frac{\langle A\,e_x\,,\,e_x\rangle}
{ \langle e_x\,,\,e_x\rangle}\end{equation} and the double
Berezin symbol of $A$ is the function defined by
\begin{equation}s(A)(x,y)=\frac {\langle A\,e_y\,,\,e_x\rangle }{ \langle e_y\,,\,e_x\rangle}
\end{equation}
for each $x,\,y\in M$ such that $\langle e_x,e_y\rangle \not=
0$, see \cite{Be1} for instance. We can easily recover $A$ from $s(A)$. Indeed, we have
\begin{align*}
A\,f(x)&=\langle A\,f\,,\,e_x \rangle  
  =\langle f\,,\,A^{\ast}\,e_x\rangle  \\
&=\int _M\,f(y)\overline {A^{\ast}\,e_x(y)}\,K(y)^{-1}\,d\mu (y) \\
&=\int _M\,f(y)\overline {\langle A^{\ast}\,e_x,e_y\rangle}\,K(y)^{-1}\,d\mu (y) \\
&=\int _M \,f(y)\,s(A)(x,y)\langle e_y,e_x\rangle\,K(y)^{-1}\,d\mu
(y).
\\ 
\end{align*}
Then we see that  the map $A\rightarrow s(A)$ is injective and that the
kernel of $A$ is the function
\begin{equation} k_A(x,y)=\langle Ae_y,e_x\rangle=s(A)(x,y)\langle
e_y,e_x\rangle.\end{equation}
Moreover, we have the following result.

\begin{proposition} \label{prop:22} \cite{Be1}, \cite{CaScand}, \cite{CGR} 
\begin{enumerate}

\item If the operator $A$ on $\mathcal H$ has adjoint $A^{\ast}$, then we
have $S(A^{\ast})=\overline{S(A)}$.

\item
For $A$ operator on $\mathcal H$ and $g\in
G$, we have \begin{equation*} S(\pi (g)^{-1}A\pi(g))(x)=S(A)(g\cdot
x),\quad \quad g\in G,\,x,\,y\in M. \end{equation*}
\end{enumerate} \end{proposition}

\section {Hilbert spaces of holomorphic functions} \label{sec:3}

Important examples of reproducing kernel Hilbert spaces are Hilbert spaces of holomorphic functions, see \cite{FaK}. Such spaces naturally appear in Harmonic Analysis as, for instance, carrying spaces of the holomorphic discrete series representations of some semi-simple Lie groups, see \cite{Kn}
or, more generally, of holomorphic representations of quasi-Hermitian Lie groups, see \cite{Ne}.

In this section, we assume that, with notation of Section \ref{sec:2}, $M$ is a domain of ${\mathbb C}^n$, $G$ consists of holomorphic automorphisms of $M$ and $\mathcal H$ consists of holomorphic functions on $M$. 

Let $\mu_L$ be the Lebesgue measure on $M\subset {\mathbb C}^n$.
Write $\mu=\delta. \mu_L$ for the $G$-invariant measure on $M$ where
$\delta>0$ is a continuous function on $M$.  In order to avoid technicalities, we also assume here that $K$ is continuous,.

Then $\mathcal H$ consists of all holomorphic functions $f$ on $M$ such that
\begin{equation*}\Vert f\Vert ^2_{\mathcal H}:=
\int _M\, \vert f(z)\vert^2K(z)^{-1}\delta(z)\,d\mu_L(z)<\infty.\end{equation*}

We easily see that the evaluation map $f\rightarrow f(z)$ is continuous.
Indeed, given $z\in M$ we can fix $r>0$ such that the closed polydisk
$D_r(z):=\{w\in M\,:\,\vert w_k-z_k\vert<r, \, k=1,2,\ldots,n\}$ is contained in $M$. By the mean value property, we have
\begin{equation*}f(z)=(\pi r^2)^{-n}\,\int_{D_r(z)}f(w)\,d\mu_L(w).
\end{equation*}
Then, by the Cauchy-Schwarz equality, we get
\begin{equation*}\vert f(z)\vert^2\leq (\pi r^2)^{-2n}\,\Vert f\Vert_{\mathcal H}^2 \int_{D_r(z)}K(w)\delta(w)^{-1}d\mu_L(w),
\end{equation*}
hence the continuity of the map $f\rightarrow f(z)$.

Note that by the same way we can show that the space ${\mathcal O}(M)$
of all holomorphic functions on $M$ being endowed with uniform convergence on compact subsets, the natural injection ${\mathcal H}\rightarrow {\mathcal O}(M)$ is injective.

In this context, the reproducing kernel $(z,w)\rightarrow k(z,w)=\langle e_w,e_z\rangle_{\mathcal H}$ of ${\mathcal H}$ is holomorphic in the variable $z$ and anti-holomorphic in the variable $w$. More generally,
let $A$ be an operator on ${\mathcal H}$. Then the function $s(A)(z,w)$
is holomorphic in the variable $z$ and anti-holomorphic in the variable $w$.
Consequently, $s(A)$-hence $A$-is determined by its restriction to the diagonal of $M\times M$, that is, by $S(A)$.

Note that, in many cases of interest, the polynomials are elements of 
${\mathcal H}$, see \cite{CaScand}, \cite{CaDS}, \cite{Ne}.

\section {Berezin symbols and spectral measures} \label{sec:4}

Here we retain the notation of Section \ref{sec:2} and we assume that we are in the setting of Section \ref{sec:3}.
Let us introduce some additional notation. Let ${\mathcal S}(\mathbb R)$ be
the space of all Schwartz functions on $\mathbb R$ and 
${\mathcal S}'(\mathbb R)$ be the space of all tempered distributions on 
$\mathbb R$. The normalization of the Fourier transform $\mathcal F:
{\mathcal S}(\mathbb R)\rightarrow {\mathcal S}(\mathbb R)$ is taken here
as follows. For $\phi\in {\mathcal S}(\mathbb R)$, we define
\begin{equation*}({\mathcal F}\phi)(x)=\int_{\mathbb R}\,e^{-itx}\phi(t)\,dt.\end{equation*}
The inverse Fourier transform is then
\begin{equation*}({\mathcal F}^{-1}\phi)(t)=\frac{1}{2\pi}\int_{\mathbb R}\,e^{itx}\phi(x)\,dx.\end{equation*}

Recall that $\mathcal F$ can be extended to ${\mathcal S}'(\mathbb R)$
via
\begin{equation*}\langle {\mathcal F}(\nu), \phi\rangle=
\langle \nu, {\mathcal F}\phi\rangle\end{equation*} 
for each $\nu \in {\mathcal S}'(\mathbb R)$ and each $\phi \in {\mathcal S}(\mathbb R)$. Similarly, one has
\begin{equation*}\langle {\mathcal F}^{-1}(\nu), \phi\rangle=
\langle \nu, {\mathcal F}^{-1}\phi\rangle.\end{equation*}

Now, let $A$ be a self-adjoint operator on $\mathcal H$ and let $A=\int_
{\mathbb R}\lambda\, dE_{\lambda}$ be the spectral decomposition of $A$.
Then we also have, for each $t\in {\mathbb R}$, 
\begin{equation*}\exp(-itA)=\int_
{\mathbb R}e^{-it\lambda} dE_{\lambda}.\end{equation*}
In particular, if we take $f\in {\mathcal H}$ such that $\Vert f\Vert_{\mathcal H}=1$ and we denote $\mu_f:=d(\langle E_{\lambda}f,f\rangle_{\mathcal H})$  then we have
\begin{equation}\label{eq:41} \langle\exp(-itA)f,f\rangle_{\mathcal H}=\int_
{\mathbb R}e^{-it\lambda} d\mu_f(\lambda).\end{equation}

Since $\Vert f\Vert_{\mathcal H}=1$, $\mu_f$ is a probability measure on $\mathbb R$ hence a tempered distribution on $\mathbb R$. Then we have the following proposition.

\begin{proposition} \label{prop:41} Let $f\in {\mathcal H}$ such that $\Vert f\Vert_{\mathcal H}=1$.

\begin{enumerate}
\item $\mu_f$ is the inverse Fourier transform of the function 
$F_f:t\rightarrow \langle\exp(-itA)f,f\rangle_{\mathcal H}$ considered as a tempered distribution on $\mathbb R$.

\item Suppose that $F_f$ is integrable on $\mathbb R$. Then $\mu_f$ is absolutely continuous with respect to the Lebesgue measure on $\mathbb R$
and its density $\varphi_f$ is given by
\begin{equation*}\varphi_f(\lambda)=
\frac{1}{2\pi}\int_{\mathbb R}\,e^{it\lambda}\,F_f(t)dt.\end{equation*}

\end{enumerate}
\end{proposition}

\begin{proof} (1) For each $\phi\in {\mathcal S}(\mathbb R)$, we have
\begin{align*}\langle {\mathcal F}\mu_f,\phi \rangle=&\langle \mu_f, {\mathcal F}\phi \rangle\\
=&\int_{\mathbb R}({\mathcal F}\phi)(\lambda)d\mu_f(\lambda)\\
=&\int_{\mathbb R}\int_{\mathbb R}e^{-it\lambda}\phi(t)\,d\mu_f({\lambda})\,dt\\
=&\int_{\mathbb R}\,F_f(t)\phi(t)\,dt
\end{align*}
by Eq. \ref{eq:41}. Then we get ${\mathcal F}\mu_f=F_f$.

(2) For each $\phi\in {\mathcal S}(\mathbb R)$, we have
\begin{align*}
\langle \mu_f,\phi \rangle=&\langle {\mathcal F}^{-1}F_f,\phi \rangle
=\langle F_f, {\mathcal F}^{-1}\phi \rangle \\
=&\int_{\mathbb R}\,F_f(t)({\mathcal F}^{-1}\phi)(t) dt \\
=&\frac{1}{2\pi}\int_{\mathbb R}\,F_f(t)\left(\int_{\mathbb R}\,e^{it\lambda}\phi(\lambda)\,d\lambda \right)dt\\
=&\frac{1}{2\pi}\int_{\mathbb R}\left ( \int_{\mathbb R}\,e^{it\lambda}\,F_f(t)dt\right)\, \phi(\lambda)d\lambda.\\
\end{align*}
Note that we can apply Fubini's Theorem, since $(t, \lambda)\rightarrow F_f(t)\phi(\lambda)$
is integrable on ${\mathbb R}^2$. The result follows.
\end{proof}

The Berezin symbols naturally appear when we consider the particular case where
$f= \Vert e_z\Vert_{\mathcal H}^{-1}e_z$ for $z\in M$. In this case we have 
\begin{equation*}F_f(t)=\Vert e_z\Vert_{\mathcal H}^{-2}\langle \exp(-itA)e_z,e_z\rangle_ {\mathcal H}=S(\exp(-itA))(z)\end{equation*}
and $\mu_f$ is then given by
\begin{equation*}\langle \mu_f,\phi \rangle=
\langle F_f,{\mathcal F}^{-1}\phi \rangle=\int_{\mathbb R}S(\exp(-itA))(z)({\mathcal F}^{-1}\phi)(t)\,dt \end{equation*}
for each $\phi \in {\mathcal S}(\mathbb R)$. In particular, if the function $t\rightarrow S(\exp(-itA))(z)$ is integrable on $\mathbb R$ for each $z\in M$ then $\mu_f$ has density
\begin{equation} \label{eq:dens} \varphi_f(\lambda)=\frac{1}{2\pi}\int_{\mathbb R}
S(\exp(-itA))(z)\,e^{it\lambda}\,dt \end{equation}
and we have the following formula for the Berezin symbol of $E_{\lambda}$
(which determines $E_{\lambda}$, see Section \ref{sec:3}). Since
\begin{equation*}S(E_{\lambda})(z)=\langle  E_{\lambda}f,f\rangle_{\mathcal H}=\mu_f(]-\infty,\lambda])=\int_{-\infty}^{\lambda}\varphi_f(x)\,dx,  \end{equation*}
we get
\begin{equation*}S(E_{\lambda})(z)=\frac{1}{2\pi}\int_{\mathbb R}
\int_{]-\infty,\lambda]}S(\exp(-itA))(z)\,e^{ixt}\,dt\,dx.\end{equation*}

In this paper, we focus on the case where $A=-id\pi(X)$ for $X\in {\mathfrak g}$. We have then
\begin{equation*}S(\exp(-itA))(z)=S(\pi(\exp(-tX)))(z).\end{equation*}
Thus we see that in this case $\mu_f$ is closely connected to the function
$g\rightarrow S(\pi(g))$ which is called \textit{ the star-exponential} (since it is a convergent version of the formal star-exponential which appears in Deformation Quantization, see \cite{Fron}, \cite{ArnExp}) and played a central role in the construction of the generalized Fourier transform, \cite{ACG}, \cite{Wild}.

\section{Example: the Heisenberg group} \label{sec:5}

Let $H$ be the Heisenberg group and
${\mathfrak h}$ be the Lie algebra of $H$. Let
$v_1,v_2,v_3$ be a basis of
${\mathfrak h}$ in which the only non trivial brackets are
$[v_1,v_2]=v_3$.

For $(a_1,a_2,a_3)\in {\mathbb R}^3$, we denote by $[a_1,a_2,a_3]$ the element $\exp_{H}(a_1v_1+a_2v_2+a_3v_3)$ of $H$. The multiplication of $H$ is then given by
\begin{equation*}[a_1,a_2,a_3]\cdot [b_1,b_2,b_3]=
[a_1+b_1,a_2+b_2,a_3+b_3+\tfrac{1}{2}(a_1b_2-a_2b_1)]\end{equation*}
for $(a_1,a_2,a_3)\in {\mathbb R}^3$ and $(b_1,b_2,b_3)\in {\mathbb R}^3$.

We fix a real number $\gamma>0$ (the
case $\gamma<0$ can be treated similarly). By the Stone-von Neumann
theorem, there exists a unique (up to unitary equivalence) unitary
irreducible representation of $H$ whose restriction to the center
of $H$ is the character $[0,0,a_3]\rightarrow e^{i\gamma a_3}$, see
\cite{Fo}. We describe now 
the Bargmann-Fock realization $\pi_{\gamma}$ of this representation, see 
for instance \cite{Tay}, \cite{CaJAM}.

Let ${\mathcal H}_{\gamma}$ be the Hilbert space of all holomorphic functions $f$
on ${\mathbb C}$ such that \begin{equation*}\Vert F\Vert^2_\gamma
:=\int_{{\mathbb C}} \vert f(z)\vert^2\, e^{-\vert
z\vert^2/2\gamma}\,d\mu_{\gamma} (z) <\infty\end{equation*} where
 $d\mu_{\gamma}(z):=(2\pi
\gamma)^{-1}\,dx\,dy$. Here $z=x+iy$ with $x, y \in {\mathbb
R}$.

Let us consider the action of of $H$ on ${\mathbb C}$
defined by $g\cdot z:=z+\gamma (a_2-ia_1)$ for $g=[a_1,a_2,a_3]\in H$ and
$z\in {\mathbb C}$. Then $\pi_{\gamma}$ is the representation of
$H$ on ${\mathcal H}_{\gamma}$ given by
\begin{equation*}(\pi_{\gamma} (g)\,f)(z)=\alpha
(g^{-1},z)\,f(g^{-1}\cdot z) \end{equation*} where  $\alpha$
is defined by
\begin{equation*}\alpha (g,z)=\exp
\left(-ia_3\gamma+(1/4) (a_2+a_1i)(-2z+\gamma(-a_2+a_1i))\right)
\end{equation*} for $g=[a_1,a_2,a_3]\in H$ and $z\in {\mathbb C}$.

The differential of $\pi_{\gamma}$ is given by
\begin{equation*}\left\{\begin{aligned}
(d\pi_{\gamma}(v_1)f)(z)=&\frac{1}{2}izf(z)+\gamma if'(z)\\
(d\pi_{\gamma}(v_2)f)(z)=&\frac{1}{2}zf(z)-\gamma f'(z)\\
(d\pi_{\gamma}(v_3)f)(z)=&i\gamma f(z).
\end{aligned}\right.\end{equation*}

The coherent states are given by $e_z^{\gamma}(w)=\exp({\bar z}w/2\gamma)$. Then we have the reproducing property
$f(z)=\langle f,e_z\rangle_{\gamma }$ for each $f\in {\mathcal
H}_{\gamma}$ where $\langle\cdot ,\cdot\rangle_{\gamma}$ denotes
the scalar product on ${\mathcal H}_{\gamma}$. Note also that an orthonormal basis of ${\mathcal H}_{\gamma}$ is the family $f_p^{\gamma}(z)=(2^p\gamma^p p!)^{-1/2}z^p$ for $p\in {\mathbb N}$.

Then we can easily verify that for each $g=[a_1,a_2,a_3]\in H$, the Berezin symbol of $\pi_{\gamma}(g)$ is 
\begin{equation*}S_{\gamma}(\pi_{\gamma}(g))(z)=
\exp\left( i\gamma a_3-\tfrac{\gamma}{4}(a_1^2+a_2^2)+
\tfrac{1}{2}(a_2+ia_1)z+\tfrac{1}{2}(-a_2+ia_1){\bar z}\right).
\end{equation*}

Now, let $X=a_1v_1+a_2v_2+a_3v_3\in {\mathfrak h}$ with $(a_1,a_2)\not=(0,0)$. Let $z=x+iy$ with $x,y \in {\mathbb R}$. Then we have
\begin{equation} \label{eq:5spiga} S_{\gamma}(\pi_{\gamma}(\exp(-tX)))(z)=e^{-it(a_1x+a_2y+a_3\gamma)}\,e^{-\gamma t^2(a_1^2+a_2^2)/4}\end{equation}
and, clearly, the function $t\rightarrow S_{\gamma}(\pi_{\gamma}(\exp(-tX)))(z)$ is integrable. Then, applying Proposition \ref{prop:41}, we obtain that the measure $$d(\langle E_{\lambda}(\Vert e^{\gamma}_z\Vert^{-1} e^{\gamma}_z),\Vert e^{\gamma}_z\Vert^{-1} e^{\gamma}_z\rangle_{\gamma})$$ associated with $-id\pi_{\gamma}(X)$ has density
\begin{align*}\varphi(&\lambda):=\frac{1}{2\pi}\int_{\mathbb R}\,e^ {it\lambda}\,S_{\gamma}(\pi_{\gamma}(\exp(-tX)))(z)\,dt \\
=&\frac{1}{\sqrt{\pi\gamma(a_1^2+a_2^2)}}
\exp\left(-\frac{1}{\gamma(a_1^2+a_2^2)}(\lambda
-a_1x-a_2y-a_3\gamma)^2\right). \end{align*}
which is a Gaussian function. The case $X=v_1$, $z=0$ (hence $e_z^{\gamma}=1$)
was already considered in \cite{Luo}.

\section{The contraction of $SU(1,1)$ to the Heisenberg group} \label{sec:6}

In this section we first recall some generalities about the holomorphic discrete series of 
$SU(1,1)$ in the context of the Berezin quantization and its contraction to the non-degenerate unitary irreducible representations of the Heisenberg group which was introduced in \cite{CaLux}. We closely follow the exposition of \cite{CaLux}, see also \cite{CaDS}.

Let $SU(1,1)$ denote the group of all matrices
\begin{equation*}g(a,b):=\begin{pmatrix}a & b\\ {\bar b} & {\bar a} \end{pmatrix}\end{equation*}  where $a,b\in {\mathbb C}$ satisfy $\vert a\vert^2-\vert b\vert^2=1$.

Note that $SU(1,1)$ naturally acts on the open unit disk ${\mathbb D}=(\vert z\vert<1)$
by fractional transforms
\begin{equation*}g(a,b)\cdot z:=\frac{az+b}{{\bar b}z+{\bar a}}. \end{equation*}

For each $z\in {\mathbb D}$ we denote $g_z:=g((1-\vert z \vert)^{-1/2},
(1-\vert z \vert)^{-1/2}z)$. Then we have $g_z\cdot 0=z$ for each $z\in {\mathbb D}$, that is, the map $z\rightarrow g_z$ is a section for the action of $SU(1,1)$ on
${\mathbb D}$.

 The Lie algebra $su(1,1)$ of $SU(1,1)$ has basis
\begin{equation*}u_1=\frac{1}{2}\begin{pmatrix}0 & -i\\ i & 0 \end{pmatrix};\quad
u_2=\frac{1}{2}\begin{pmatrix}0 & 1\\ 1 & 0 \end{pmatrix};\quad
u_3=\frac{1}{2}\begin{pmatrix}-i & 0\\ 0 & i \end{pmatrix}.
\end{equation*}

We introduce now the holomorphic discrete series $(\pi_n)$ of $SU(1,1)$.
Fix an integer $n>2$.  Let ${\mathcal H}_n$ be the Hilbert space of all holomorphic functions $f:{\mathbb D}\rightarrow {\mathbb C}$ such that

\begin{equation*}\Vert f\Vert_n ^2=\int_{\mathbb D}\,\vert f(z) \vert ^2\,d\mu_n(z)\,<\infty \end{equation*} 
where $d\mu_n(z):=\frac{n-1}{ \pi} (1-z{\bar
z})^{n-2}\,dx\,dy$, $dx\,dy$ denoting as usual the Lebesgue measure on
${\mathbb C}\simeq {\mathbb R}^2$.

An orthonormal basis of ${\mathcal
H}_n$ is then given by the family 
$f^n_p\,(z)=\binom{n+p-1}{p}^{1/2}\,z^p$ for $p\in {\mathbb N}$ 
and the coherent states are $e_z^n(w)=(1-w{\bar
z})^{-n}$ for $z,w\in {\mathbb D}$.

Let $\pi_n$ be the representation of $SU(1,1)$ defined on ${\mathcal
H}_n$ by 
\begin{equation*}(\pi_n\,(g(a,b))\,f)(z)=(a-{\bar b}z)^{-n}\,f(g(a,b)^{-1}\cdot
z).\end{equation*}
Then the family $(\pi_n)$ is the holomorphic discrete series of $SU(1,1)$, see \cite{Kn}.

The differential $d\pi_n$ is given by
\begin{equation*}\left\{\begin{aligned}
(d\pi_{n}(u_1)f)(z)=&\frac{n}{2}izf(z)+\frac{1}{2}i(z^2+1) f'(z)\\
(d\pi_{n}(u_2)f)(z)=&\frac{n}{2}zf(z)+\frac{1}{2}(z^2-1) f'(z) \\
(d\pi_{n}(u_3)f)(z)=&\frac{n}{2}if(z)+iz f'(z).
\end{aligned}\right.\end{equation*}

For each operator $A$ on ${\mathcal
H}_n$ we denote by $S_n(A)$ the Berezin symbol of $A$. Then we have 

\begin{equation}\label{eq:6spi} S_n(\pi_n(g(a,b)))(z)=(a-{\bar a}z{\bar z}-{\bar b}z+b{\bar z})^{-n}\,(1-z{\bar z})^{n},\end{equation}
see \cite{CaLux}.

Now we  introduce the contraction of $SU(1,1)$ to the Heisenberg
group at the Lie algebra level. 
Let $r>0$ and let $C_r:{\mathfrak h}\rightarrow su(1,1)$ be the linear map defined by
\begin{equation*}C_r(v_1)=r u_1 ,\quad C_r(v_2)=r u_2,\quad
C_r(v_3)=r^2 u_3.\end{equation*}

Then we have for each $X,Y\in {\mathfrak h}$
\begin{equation*}\lim_{r\to 0}\,C^{-1}_r([C_r(X)\,,\,C_r(Y)]_{su(1,1)})=[X,Y]_{\mathfrak h}.\end{equation*} 
and we say that the family
$(C_r)_{r>0}$ is a contraction of $su(1,1)$ to ${\mathfrak h}$, see
\cite{IW}, \cite{MN}, \cite{Doo}.

The corresponding group contraction $c_r:H\rightarrow SU(1,1)$ is then given by 
\begin{equation*}c_r(\exp_H\,X)=\exp_{SU(1,1)}(C_r\,(X))\end{equation*}
and satisfies the following property: for each $x,y\in H$ there exists $r_0>0$ such that, for each $r>0$ such that $r<r_0$, the expression  $c^{-1}_r\,(c_r\,(x)c_r\,(y)^{-1})$ is well-defined and we have
\begin{equation*}\lim_{r\to 0}\,c^{-1}_r\,(c_r\,(x)\,c_r\,(y)^{-1})=xy^{-1}, \end{equation*} see \cite{CaLux}.

The following proposition was proved in \cite{CaLux}. For each $n>2$ let $r(n)>0$ such that $nr(n)^2=2\gamma$. A geometric interpretation of this
quite mysterious condition (in terms of coadjoint orbits associated with representations) can be found in \cite{CaLux}, see also \cite{CaDS}.

\begin{proposition} \label{prop:contrep} Let $h=[a_1,a_2,a_3]\in H$ and let 
\begin{equation*} g_n:=
c_{r(n)}(h)=\exp_{SU(1,1)}(C_{r(n)}(a_1v_1+a_2v_2+a_3v_3)).
\end{equation*}
 Then 
\begin{enumerate}
\item   For each $z\in {\mathbb C}$, we have
\begin{equation*}\lim_{n\to +\infty}
S_n({\pi}_n(g_n))\left(\frac{z}{\sqrt{2\gamma
n}}\right)=s_{\gamma}(\pi_{\gamma}(h))(z).\end{equation*}

\item For each $p,q \in {\mathbb N}$, we have
\begin{equation*}\lim_{n\to +\infty}\langle \pi_n(g_n)f_p^n,f_q^n\rangle_n=
\langle \pi_{\gamma}(h)f_p^{\gamma},f_q^{\gamma}\rangle_{\gamma}.
\end{equation*}

\item For each $n>2$, let $B_n:{\mathcal H}_{\gamma}\rightarrow {\mathcal H}_n$ be the unitary operator defined by $B_n(f_p^{\gamma})=f_p^n$ for each
$p\in  {\mathbb N}$. For each $f\in  {\mathcal H}_{\gamma}$, we have

\begin{equation*}\lim_{n\to +\infty}\Vert (B_n^{-1}\pi_n(g_n)B_n)f-\pi_{\gamma}(h)f\Vert_{\gamma}=0.
\end{equation*}
\end{enumerate}

\end{proposition}

\begin{proof} Here we just detail the proof of (1) since it is of some interest for our purpose. 

Let $a \in {\mathbb R}$, $\beta \in {\mathbb C}$ and let $R\in {\mathbb R}\cup i{\mathbb R}$
such that $R^2=-a^2+\vert \beta \vert^2$. Then we have
\begin{equation} \label{eq:exp}\exp\begin{pmatrix}ai & \beta \\ \bar{\beta} & -ai \end{pmatrix}=g\left( \cosh R+\frac{\sinh R}{R}ai, \frac{\sinh R}{R}\beta\right).
\end{equation}
From this we deduce that if we denote $R(n)=\frac{1}{2}r(n)(a_1^2+a_2^2
-r(n)^2a_3^2)^{1/2}$ then we have $g_n=g(\alpha_n, \beta_n)$ with

\begin{equation*}\left\{\begin{aligned}
\alpha_n=&\cosh R(n)-ir(n)^2a_3\frac{\sinh R(n)}{2R(n)};\\
\beta_n=& r(n)(a_2-ia_1)\frac{\sinh R(n)}{2R(n)}.\\
\end{aligned}\right.\end{equation*}

Also, by Eq. \ref{eq:6spi}, we have
\begin{equation*}S_n(\pi_n(g_n))\left(\frac{z}{\sqrt {2\gamma n}}\right)=\left(\alpha_n-{\bar {\alpha_n}}\frac{\vert z\vert^2}{2\gamma n}-{\bar \beta_n}\frac{z}{\sqrt{2\gamma n}}+\beta_n\frac{\bar z}{\sqrt{2\gamma n}}\right)^{-n}\,\left(1-\frac{\vert z\vert^2}{2\gamma n}\right)^{n}.\end{equation*} 

Then 
\begin{equation*}
\log\left(S_n(\pi_n(g_n))\left(\frac{z}{\sqrt {2\gamma n}}\right)\right)\\
\sim -n(\alpha_n-1) +n{\bar \beta_n}\frac{z}{\sqrt {2\gamma n}}-n\beta_n
\frac{\bar z}{\sqrt {2\gamma n}}.
\end{equation*}

Consequently, since  
\begin{equation*}n(\alpha_n-1)=n\left(\cosh R(n)-1-ir(n)^2a_3 \frac{\sinh R(n)}{2R(n)}\right)\end{equation*}
has limit $\tfrac{\gamma}{4}(a_1^2+a_2^2)-i\gamma a_3$ and $n\beta_n/\sqrt{2\gamma n}$ has limit $a_2-ia_1$ when $n\to \infty$,
we see that 
\begin{equation*}\lim_{n\to +\infty} \,S_n(\pi_n(g_n))\left(\frac{z}{\sqrt {2\gamma n}}\right)=\exp \left( -\tfrac{\gamma}{4}(a_1^2+a_2^2)+i\gamma a_3
+\tfrac{1}{2}(a_2+ia_1)z-\tfrac{1}{2}(a_2-ia_1){\bar z}\right). \end{equation*}

The result hence follows by Eq. \ref{eq:5spiga}.
\end{proof}

For each $z\in {\mathbb C}$ let $f_z^{\gamma}:=\Vert e_z^{\gamma}\Vert_{\gamma}^{-1}e_z^{\gamma}\in {\mathcal H}_{\gamma}$
and for each $z\in {\mathbb D}$ let $f_z^{n}:=\Vert e_z^{n}\Vert_{n}^{-1}e_z^{n}\in {\mathcal H}_{n}$.

Let $X=a_1v_1+a_2v_2+a_3v_3\in {\mathfrak h}$ with $(a_1,a_2)\not=(0,0)$.
In accordance with the notation of Section \ref{sec:4}, we denote by $\varphi_{f_z^{\gamma}}$ the density of the spectral measure $\mu_{f_z^{\gamma}}$ corresponding to $-id\pi_{\gamma}(X)$ and by 
$\varphi_{f_z^{n}}$ the density of the spectral measure $\mu_{f_z^{n}}$ corresponding to $-id\pi_{n}(C_{r(n)}(X))$, see Section \ref{sec:4} and Section \ref{sec:5}. Then we have the following contraction result for these densities.

\begin{proposition} For each $\lambda \in {\mathbb R}$ and each $z\in {\mathbb C}$, we have 
$\lim_{n\to +\infty} \varphi_{f_{z/\sqrt{2\gamma n}}^{n}}(\lambda)=\varphi_{f_z^{\gamma}}(\lambda)$.
\end{proposition}

\begin{proof} By Eq. \ref{eq:dens}, we have
\begin{equation*} \varphi_{f_{z/\sqrt{2\gamma n}}^{n}}(\lambda)=\frac{1}{2\pi}\int_{\mathbb R}
S_n(\pi_n(\exp(tC_{r(n)}(X))(z/\sqrt{2\gamma n})\,e^{-it\lambda}\,dt, \end{equation*}
so, taking into account (1) of Proposition \ref{prop:contrep}, we see that in order to get the result we have just to verify that the dominated convergence theorem
can be applied. To this end, we first note that $z/\sqrt{2\gamma n}\in {\mathbb D}$ for $n$ large enough. Then the expression $S_n(\pi_n(\exp(tC_{r(n)}(X))(z/\sqrt{2\gamma n})$ is well-defined for $n$ large enough and we have
\begin{align*}S_n(\pi_n(\exp(tC_{r(n)}(X))))(z/\sqrt{2\gamma n})=&
S_n(\pi_n(g_{z/\sqrt{2\gamma n}})^{-1}\pi_n(\exp(tC_{r(n)}(X)))
\pi_n(g_{z/\sqrt{2\gamma n}})))(0)\\
=&S_n(\pi_n(\exp(t\Ad(g_{z/\sqrt{2\gamma n}})^{-1}C_{r(n)}(X))))(0) \end{align*}
by (2) of Proposition \ref{prop:22}.

Now, let us denote by $(b^n_{ij})_{1\leq i,j\leq 3}$ the matrix of $\Ad(g_{z/\sqrt{2\gamma n}})^{-1}:su(1,1)\rightarrow su(1,1)$ in the basis $(u_i)_{1\leq i\leq 3}$ and introduce
\begin{equation*}\left\{\begin{aligned}
c_1^n:=&a_1b_{11}^n+a_2b_{12}^n+r(n)a_3b_{13}^n\\
c_2^n:=&a_1b_{21}^n+a_2b_{22}^n+r(n)a_3b_{23}^n \\
c_3^n:=&a_1b_{31}^n+a_2b_{32}^n+r(n)a_3b_{33}^n.
\end{aligned}\right.\end{equation*}

Then we have
\begin{align*} \Ad(g_{z/\sqrt{2\gamma n}})^{-1}(C_{r(n)}(X))&=
\Ad(g_{z/\sqrt{2\gamma n}})^{-1}(r(n)a_1u_1+r(n)a_2u_2+r(n)^2a_3u_3)\\
&=r(n)(c_1^nu_1+c_2^nu_2+c_3^nu_3).\end{align*}

Note that the sequences $(c_1^n)$, $(c_2^n)$ and $(c_3^n)$ are convergent; we denote by $c_1,c_2$ and $c_3$ the corresponding limits. Since we have
\begin{equation*}(c_1^n)^2+(c_2^n)^2-(c_3^n)^2=a_1^2+a_2^2-r(n)^2a_3^2,\end{equation*} we get
\begin{equation*}c_1^2+c_2^2-c_3^2=a_1^2+a_2^2>0.\end{equation*}

We denote $d_n:=((c_1^n)^2+(c_2^n)^2-(c_3^n)^2)^{1/2}$ for $n$ large enough. Then, by Eq. \ref{eq:6spi} and Eq. \ref{eq:exp}, we obtain
\begin{equation*}S_n(\pi_n(\exp(t\Ad(g_{z/\sqrt{2\gamma n}})^{-1}C_{r(n)}(X))))(0)=\left( \cosh( \tfrac{1}{2}r(n)d_nt)-ic_3^nd_n^{-1}\sinh( \tfrac{1}{2}r(n)d_nt)\right)^{-n}.\end{equation*}

Hence
\begin{align*}\vert &S_n(\pi_n(\exp(t\Ad(g_{z/\sqrt{2\gamma n}})^{-1}C_{r(n)}(X))))(0) \vert \leq (\cosh( \tfrac{1}{2}r(n)d_nt))^{-n}\\
&\leq \left( 1+ \tfrac{1}{2}(\tfrac{1}{2}r(n)d_nt)^2\right)^{-n}
\leq \left( 1+\frac{\gamma}{4n}d_n^2t^2\right)^{-n}. \end{align*}

Finally, since there exists $C>0$ such that $d_n\geq C$ for each $n$ large enough,
we conclude that there exists $C'>0$ such that 
\begin{equation*}\vert S_n(\pi_n(\exp(t\Ad(g_{z/\sqrt{2\gamma n}})^{-1}C_{r(n)}(X))))(0) \vert \leq \left(1+\frac{C't^2}{n}\right)^{-n}
\leq (1+C't^2)^{-1} \end{equation*}
for each $n$ sufficiently large. The result follows.
\end{proof}

The case $X=v_1$, $z=0$ considered in \cite{Luo} corresponds to the limit
\begin{equation*} \lim_{n\to +\infty} \frac{1}{2\pi}\, \int_{\mathbb R}\,e^{-it\lambda} \left(\cosh \tfrac{1}{2}r(n)t\right)^{-n}dt=
\frac{1}{2\pi}\, \int_{\mathbb R}\,e^{-it\lambda}\,e^{-\gamma t^2/4}dt
=\frac{1}{\sqrt{\pi\gamma}}e^{-{\lambda}^2/\gamma}.\end{equation*}

\section{The contraction of $SU(2)$ to the Heisenberg group} \label{sec:7}

The contraction of the unitary irreductible representations of $SU(2)$ to the unitary irreducible representations of $H$ was investigated in \cite{Ri} and \cite{CaBe1}. Some applications to Fourier multipliers can be found in \cite{DooG}. This contraction is quite analogous to that of the previous section but a little bit more complicated since the unitary irreductible representations of $SU(2)$ are finite-dimensional while the (non-degenerate) unitary irreducible representations of $H$ are not.

Let us denote the elements of $SU(2)$ as  \begin{equation*}g(a,b):=\begin{pmatrix}a & b\\ -{\bar b} & {\bar a} \end{pmatrix}, \quad a,b\in {\mathbb C}, \quad \vert a\vert^2+\vert b\vert^2=1.\end{equation*}

The Lie algebra $su(2)$ of $SU(2)$ has basis
\begin{equation*}u'_1=\frac{1}{2}\begin{pmatrix}0 & i\\ i & 0 \end{pmatrix};\quad
u'_2=\frac{1}{2}\begin{pmatrix}0 & -1\\ 1 & 0 \end{pmatrix};\quad
u'_3=\frac{1}{2}\begin{pmatrix}i & 0\\ 0 & -i \end{pmatrix}.
\end{equation*}

For each integer $m>0$, we denote by ${\mathcal F}_m$ the space of all complex polynomials of degree $\leq m$ endowed with the Hilbertian norm

\begin{equation*}\Vert f\Vert_m ^2=\int_{\mathbb C}\,\vert f(z) \vert ^2\,\,\tfrac{m+1}{ \pi} (1+z{\bar
z})^{-m-2}\,dx\,dy. \end{equation*} 

Then ${\mathcal F}_m$ is a Hilbert space of dimension $m+1$ (its Hilbert product is denoted by $\langle \cdot,\cdot \rangle_m$) and 
an orthonormal basis of ${\mathcal
F}_m$ is the family 
$f^m_p\,(z)={\binom{m}{p}}^{1/2}\,z^p$ for $p=0,1,\ldots, m$.

For each $w,z \in {\mathbb C}$, let $e_z^m(w)=(1+w{\bar
z})^{m}$. Then we have the reproducing property $\langle f,e_z^m\rangle_m=f(z)$ for each $f\in {\mathcal
F}_m$ and $z \in {\mathbb C}$.

We define the representation $\rho_m$ of $SU(2)$ on ${\mathcal
F}_m$ by 
\begin{equation*}(\rho_m\,(g(a,b))\,f)(z)=(a+{\bar b}z)^{m}\,f\left(\frac{{\bar a}z-b}{{\bar b}z+a}\right).\end{equation*}
Then $\rho_m$ is a unitary irreductible representation of $SU(2)$ whose differential 
is given by
\begin{equation*}\left\{\begin{aligned}
(d\rho_m(u'_1)f)(z)=&-\frac{m}{2}izf(z)+\frac{1}{2}i(z^2-1) f'(z)\\
(d\rho_m(u'_2)f)(z)=&-\frac{m}{2}zf(z)+\frac{1}{2}(z^2+1) f'(z) \\
(d\rho_m(u'_3)f)(z)=&\frac{m}{2}if(z)-iz f'(z).
\end{aligned}\right.\end{equation*}

For each operator $A$ on ${\mathcal
F}_m$, we denote by $S_m(A)$ the Berezin symbol of $A$.
Then we can verify that, see \cite{CaBe1},

\begin{equation}\label{eq:7spi} S_m(\rho_m(g(a,b)))(z)=(a+{\bar a}z{\bar z}+{\bar b}z-b{\bar z})^{m}\,(1+z{\bar z})^{-m}.\end{equation}

For each $r>0$, let $C'_r:{\mathfrak h}\rightarrow su(2)$ be the linear map defined by
\begin{equation*}C'_r(v_1)=r u'_1 ,\quad C'_r(v_2)=r u'_2,\quad
C'_r(v_3)=r^2 u'_3.\end{equation*}

Then we can verify that $(C'_r)$ is a contraction of $su(2)$ to $\mathfrak h$
the corresponding contraction of $SU(2)$ to $H$ being given by
\begin{equation*}c'_r(\exp_H\,X)=\exp_{SU(2)}(C'_r\,(X))\end{equation*}
for each $X\in {\mathfrak h}$, see \cite{Ri}.

For each integer $m>0$, let $r(m)>0$ be such that $mr(m)^2=2\gamma$.
Also, let $B'_m:{\mathcal F}_m\rightarrow {\mathcal F}_m\subset {\mathcal H}_{\gamma}$ be the unitary operator defined by $B'_m(f_p^m)=f_p^{\gamma}$. Then we have the following result which is analogous to Proposition \ref{prop:contrep}.

\begin{proposition} \label{prop:contrepbis} \rm{\cite{CaBe1}} Let $h=[a_1,a_2,a_3]\in H$ such that $(a_1,a_2)\not=(0,0)$ and for each $m>0$, let 
\begin{equation*} g_m:=
c'_{r(m)}(h)=\exp_{SU(2)}(C'_{r(m)}(a_1v_1+a_2v_2+a_3v_3)).
\end{equation*}
 Then 
\begin{enumerate}
\item   For each $z\in {\mathbb C}$, we have
\begin{equation*}\lim_{m\to +\infty}
S_m({\rho}_m(g_m))\left(\frac{z}{\sqrt{2\gamma
m}}\right)=s_{\gamma}(\pi_{\gamma}(h))(z).\end{equation*}

\item For each $p,q \in {\mathbb N}$, we have
\begin{equation*}\lim_{n\to +\infty}\langle \rho_m(g_m)f_p^m,f_q^m\rangle_m=
\langle \pi_{\gamma}(h)f_p^{\gamma},f_q^{\gamma}\rangle_{\gamma}.
\end{equation*}

\item For each polynomial $f\in  {\mathcal H}_{\gamma}$, we have

\begin{equation*}\lim_{m\to +\infty}\Vert (B'_m\rho_m(g_m)B'^{-1}_m)f-\pi_{\gamma}(h)f\Vert_{\gamma}=0.
\end{equation*} (This makes sense since we then have $f\in {\mathcal F}_m$ for each $m$ large enough).
\end{enumerate}
\end{proposition}

For each $z\in {\mathbb C}$ let $f_z^{m}:=\Vert e_z^{m}\Vert_{m}^{-1}e_z^{m}\in {\mathcal F}_{m}$
and for each $X=a_1v_1+a_2v_2+a_3v_3\in {\mathfrak h}$ with $(a_1,a_2)\not=(0,0)$, let us denote by $\mu_{f_z^{m}}$
the spectral measure $d(\langle E_{\lambda}f_z^{m},f_z^{m}\rangle_m)$
corresponding to $-id\rho_{m}(C_{r(m)}(X))$ (see  Section \ref{sec:4}).

\begin{proposition} \label{prop:cvsp2}
The sequence   $\mu_{f_{z/\sqrt{2\gamma m}}^{m}}$ converges to $\mu_{f_z^{\gamma}}$ in ${\mathcal S}'({\mathbb R})$. \end{proposition}

\begin{proof} Let $z\in {\mathbb C}$. To simplify notation, let $F(t):=S_{\gamma}(\pi_{\gamma}(\exp(tX)))(z)$ and for each $m\in {\mathbb N}$, let $F_m(t):=S_{m}(\rho_{m}(\exp(tC'_{r(m)}(X))))(z/\sqrt{2\gamma m})$. By (1) of Proposition \ref{prop:contrepbis},
we have $\lim_{m\to +\infty} F_m(t)=F(t)$ for each $t\in {\mathbb R}$.

Moreover, by the Cauchy-Schwarz inequality, we have
\begin{align*}\vert F_m&(t)\vert =\vert \langle \rho_{m}(\exp(tC'_{r(m)}(X)))
f_z^{m},f_z^{m}\rangle_m\vert \\
&\leq \Vert \rho_{m}(\exp(tC'_{r(m)}(X)))
f_z^{m}\Vert_m\, \Vert f_z^{m}\Vert_m \leq 1\\
\end{align*} since $\rho_{m}$ is unitary.

This implies that $(F_m)$ converges to $F$ in ${\mathcal S}'(\mathbb R)$. Indeed, for each $\phi \in {\mathcal S}(\mathbb R)$, the dominated convergence theorem shows that $\lim_{m\to +\infty}\int_{\mathbb R}F_m\phi=\int_{\mathbb R}F\phi$
since $\vert F_m\phi\vert \leq \vert \phi\vert\in L^1(\mathbb R)$ for each $m\geq 0$.

Finally, since ${\mathcal F}^{-1}:{\mathcal S}'(\mathbb R)\rightarrow {\mathcal S}'(\mathbb R)$ is
continuous, the result follows from (1) of Proposition \ref{prop:41}. 
\end{proof}

We conclude with the following example. We consider the case where $X=v_1$, $z=0$ hence $f_z^m=1$. 

First, let us compute the spectral measure $\mu^1:=d(\langle E_{\lambda}1,1\rangle_m)$ corresponding to the operator $A_1:=-id\rho_m(u'_1)$ on ${\mathcal F}_m$.

By Proposition \ref{prop:41}, we have just to compute the inverse Fourier transform of the function 
\begin{equation*}S_m(\rho_m(\exp(-tu'_1)))(0)=\left(\cos \left(\tfrac{1}{2}t \right)\right)^m=\left(\frac{1}{2}\right)^m\sum_{k=0}^m\binom{m}{k}
e^{it\left(\tfrac{m}{2}-k\right)}.\end{equation*}

But we can easily verify that for each $\lambda \in {\mathbb R}$, we have ${\mathcal F}(\delta_{\lambda})=e^{-i\lambda t}$. Consequently, we get
\begin{equation*} \mu^1=\left(\frac{1}{2}\right)^m\sum_{k=0}^m\binom{m}{k}
\delta_{k-\tfrac{m}{2}}.\end{equation*}

In fact, we can also find this result directly but this is a little bit more longer.
This can be done as follows. First, we remark that a basis of ${\mathcal F}_m$ consisting of eigenvectors of $A_1$ is \begin{equation*}F_k^m=(z-1)^k(z+1)^{m-k}, \qquad
k=0,1,\ldots m, \end{equation*}
the eigenvalue associated with $F_k^m$ being $\lambda_k:=k-\tfrac{m}{2}$. 

Then, denoting by $P_k$ the orthogonal projection operator of ${\mathcal F}_m$ on the line generated by $F_k^m$, we have
\begin{equation*} \mu^1=\sum_{k=0}^m\langle P_k(1),1\rangle_m\delta_{\lambda_k}.\end{equation*}

It remains to compute $\langle P_k(1),1\rangle_m$. One has 
$P_k(1)=\Vert F_k^m\Vert_m^{-2}\,\langle 1,F_k^m\rangle_mF_k^m$
hence 
\begin{equation*}\langle P_k(1),1\rangle_m=\Vert F_k^m\Vert_m^{-2}\vert  \langle 1,F_k^m\rangle_m\vert^2. \end{equation*}

Now, on the one hand, we have $ \langle 1,F_k^m\rangle_m=(-1)^k$ since
$ \langle 1,z^q\rangle_m=0$ for each $q>0$. On the other hand,
by the binomial formula we have  
\begin{equation*}2^m=\sum_{k=0}^m\binom{m}{k}(-1)^kF_k^m
\end{equation*} hence 
\begin{equation*} \langle 1,F_k^m\rangle_m=\left(\frac{1}{2}\right)^m
\binom{m}{k}(-1)^k\Vert F_k^m\Vert_m^{2}\end{equation*} which gives
$\Vert F_k^m\Vert_m^{2}=2^m {\binom{m}{k}}^{-1}$ and finally
$\langle P_k(1),1\rangle_m=\left(\tfrac{1}{2}\right)^m {\binom{m}{k}}$
as required.

Similarly, we can verify that the spectral measure corresponding to the operator 
\begin{equation*}-id\rho_m(C_{r(m)}(v_1))=-ir(m)d\rho_m(u'_1)
\end{equation*} is 
\begin{equation*}\mu_m:=\left(\frac{1}{2}\right)^m \sum_{k=0}^m{\binom{m}{k}}\delta_{r(m)(k-\tfrac{m}{2})}.\end{equation*}

Consequently, applying Proposition \ref{prop:cvsp2} and taking the results
of Section \ref{sec:5} into account, we obtain the following result.

\begin{proposition} For each $\phi \in {\mathcal S}({\mathbb R})$, we have
\begin{equation*}\lim_{m\to +\infty}\left(\frac{1}{2}\right)^m \sum_{k=0}^m{\binom{m}{k}}\phi \left(\sqrt{\frac{2\gamma}{m}}(k-\tfrac{m}{2})\right)=\int_{\mathbb R}\, \frac{1}{\sqrt{\pi \gamma}}\,e^{-\lambda^2/\gamma}\phi (\lambda)\,d\lambda.
\end{equation*}
 
\end{proposition}

\noindent \textbf {Acknowledgements.} I would like to thank the referee for suggesting some improvements to the text of the paper.

\noindent \textbf {Data availability statement.} Data sharing is not applicable to this article as no datasets were generated or analysed during the current study.

\end{document}